\documentclass[a4paper,10pt]{article}

\usepackage[all]{xy}
\usepackage[utf8]{inputenc}
\usepackage{amsmath,amsfonts,amssymb,amsthm,epsfig,epstopdf,titling,url,array}

\theoremstyle{plain}
\newtheorem{theorem}{Theorem}[section]
\newtheorem{lemma}[theorem]{Lemma}
\newtheorem{proposition}[theorem]{Proposition}

\newtheorem{definition}{Definition}[section]
\newtheorem{remark}{Remark}[section]
\newtheorem{assertion}{\bf Theorem}[]
\usepackage[svgnames,dvipsnames,x11names]{xcolor}

\usepackage{amssymb}
\usepackage{amsthm}
\usepackage[american]{babel}
\usepackage{caption}
\usepackage[OT1,T1]{fontenc}
\usepackage{makeidx}
\usepackage{mathptmx}
\usepackage{multicol}
\usepackage{rotating}
\usepackage{subfigure}
\usepackage{upgreek}
\usepackage{xcolor}
\usepackage{lmodern}
\usepackage{textcomp}

\newcommand{\R}{\mathbb{R}}

\newcommand{\M}{\overline{M}}

\def\lra{\longrightarrow}

\def\dg{ \text{deg}}

\def\sph{\mathbb{S}}

\def\CC{\mathbb{C}}

\newtheorem{question}{\textbf{Question}}[]

\theoremstyle{definition}

\date{2018-12-22}

\begin{document}

\title{On the Gauss map of  finite\\ 
geometric type surfaces.}
\author{N. A. de Andrade and L. P. Jorge}
\date{}

\maketitle


\begin{abstract} 
Surfaces of finite geometric type are complete, immersed into the 
tree-dimensional Euclidean space with finite total curvature and Gauss map 
extending to an oriented compact surface as a smooth branched covering map over 
the unit sphere of the Euclidean three dimensional space. In a recent preprint J. Jorge and F. Mercuri gave a geometric 
proof that the Gauss map can not omit three or more points if the immersion is 
minimal and no flat. Here we give a topological proof of this result in the 
class of no flat finite geometric type surfaces and also give a topological 
classification when the Gauss map is a regular covering map. This facts are 
easy applications of our main result, a generalization of the little Picard 
theorem for the class of branched covering of a finite geometric type surface 
into the unit sphere of the tree dimensional Euclidean space. A  
finite geometric type surface given by a compact surface minus a finite set of 
points has the following property: any branched covering from the surface to 
the unit Euclidean sphere having a $C^0$ extension to the compact surfaces 
can miss at most $2$ points. This is a generalization of the little Picard 
theorem to the class of finite geometric type surfaces.
\end{abstract}


\section{Introduction}  \hspace{.5cm}

Finite geometric type surfaces was introduced in
 \cite{kn:B-F-M}\ as those immersions $\varphi \colon 
M\to \mathbb{R}^{3}$\ of a surface $M$\ such that $M$\ is complete in the 
induced metric and
\begin{enumerate}
\item $M$\ is diffeomorphic to a compact oriented surface $\overline{M}$\ 
minus a finite set of points, $E_m = \{w_1, \ldots, w_m\}$,
\item the Gaussian curvature vanishes only at a finite number of points,
\item the Gauss map $G$\ extends to a smooth branched covering, denoted by the 
same symbol, $G \colon \M\to \mathbb{S}^2$.
\end{enumerate}
The points $w_i$, or sometimes punctured neighborhoods of these points are 
called the {\em ends} on $\varphi$. The authors in \cite{kn:B-F-M} proved that 
the cardinality of $\sph^2\setminus G(M)$ is at most $3$ like the minimal case 
proved by Osserman \cite{kn:O}. In Jorge and Mercuri preprint 
\cite{kn:J-Mercuri} the authors shows that in the subclass of minimal 
immersions into $\R^3$ of finite geometric type the number of missing points of 
the Gauss map is at most $2$ and this number is sharp. In the paper of L. 
Rodriguez \cite{kn:Rod} the author consider the classification of minimal 
immersions of finite geometric type that are embedded. He 
answer affirmatively a particular case of the following questions: 
\begin{enumerate}
 \item[(Q1)] Is the catenoid the only embedded one?
 \item[(Q2)] If the Gauss map of such surface omits two or more points, then 
must it be a covering of the catenoid?
\item[(Q3)] If the Gaussian curvature is always strictly negative then must the 
surface be catenoid?
\end{enumerate}
He proves the following result.
\begin{theorem}[\cite{kn:Rod}]
 If the surface is minimal, embedded with Gaussian curvature strictly negative 
and of finite geometric type then it is the catenoid. 
\end{theorem}
We will consider in this notes finite geometric type surfaces whose Gauss map 
omit two or more points and is a covering map, that is, the Gaussian curvature 
has no zeros. Before we go further let's introduce the examples of \cite{kn:M-S} 
and \cite{kn:G-J}. In \cite{kn:M-S} the authors present examples of minimal 
immersions of $\sph^2\setminus E_4$ whose image of Gauss map is exactly 
$\sph^2 / \{a,b\},\ a, b\in \sph^2$ and of tori $T^2\setminus E_3$ where the 
Gauss map image is $\sph^2\setminus \{\pm a\},\ a\in \sph^2.$ In \cite{kn:G-J} 
there are examples of families $\M\setminus E_4$ into $\R^3$ with the Gauss map 
omitting a par of antipodal points of $\sph^2$ for any genus. By other side we 
can bend a minimal immersion of finite geometric type preserving the 
cardinality of the set of omitted points by the Gauss map. If $M$ is a finite 
geometric type surface with non empty set $Y=\sph^2\setminus G(M)$, we define 
\begin{equation}
 Y_\epsilon=\bigcup_{y\in Y} B_\epsilon^{\sph^2}(y),\qquad 
0<\epsilon<\min\big\{ \text{dist}^{\sph^2}(y_1,y_2)\ \big|\ y_1,\ y_2\in Y,\ 
y_1\neq y_2 \big\}.
\end{equation}
If $\sharp Y=1$ take $Y_\epsilon$ as one small ball inside the image of 
neighborhood of ends that are graphic over an unbounded annulus. 
In section (\ref{ftp}) we prove that it is possible to bend $M$ to get new 
immersion $\tilde M$ satisfying the following.
\begin{assertion}
 Let $M$ be a finite geometric type surface with set $Y=\{y_j\ |\ 1\leq j\leq 
3\}$ of omitted points by the Gauss map $G$ of $M.$ Then there is 
$\epsilon_0>0$ such that for each $0<\epsilon\leq\epsilon_0$ and $\tilde 
Y=\{\tilde y_j\}$ with $\tilde y_j$ in the connected component of $Y_\epsilon$, 
 exist one finite geometric type surface $\tilde M$ with Gauss map $\tilde G$ 
omitting exactly $\tilde Y.$ The surface $\tilde M$ is one bent of $M.$
\end{assertion}
Therefore it is possible to give only topological classification of finite 
geometric type surfaces. For example, assuming that the 
Gauss map of a finite geometric type surface $M$ is a covering map, or 
equivalently, the Gaussian curvature has no zero on $M$ our conclusion is the 
following. 
\begin{assertion}
 Let $M$ be an non flat finite geometric type surface with the Gauss map $G$ a 
covering map. Then
\begin{enumerate}
 \item There is no example if $\sph^2\setminus G(M)$ has tree points,
 \item If $\sph^2\setminus G(M)$ has two points then $M$ is a covering map of 
$\sph^2\setminus \{a,\ b\},\ a\neq b.$
\item If $\sph^2\setminus G(M)$ has one point then $M$ is diffeomorphic to 
$\CC.$ 
\end{enumerate}
\end{assertion}
The items $2$ and $3$ are manipulations of Riemann-Hurwirtz {\bf (RH)} and total
 curvature  {\bf 
(TC)} formulas. The proof of item $1$ is included  in the next result.
In fact we give a topological prove 
of the same result of \cite{kn:J-Mercuri} but for the class of all no flat 
finite geometric type immersions. The idea follows like 
this. First we prove 
that the cardinality $\sharp(\sph^2\setminus F(M))\leq3$ for any branched 
covering $F\colon M\to\sph^2$ having a $C^0$ extension to a branched map 
$F\colon\M\to\sph^2$ (see theorem \ref{th3.2}).  Given a set $Y\subset\sph^2,\ 
\sharp Y=3,$ it is possible to find a branched covering map 
$f\colon\sph^2\to\sph^2$ of degree $4$ such that $f^{-1}(Y)=X,\ \sharp X=6,$ and 
for each $y\in Y$ the set $f^{-1}(y)$ has two points, one with order of 
branching $0$ and the other with order $2$  and the 
induced map $f_\star$ 
between the homotopy groups $\pi_1(\sph^2\setminus X)$ to 
$\pi_1(\sph^2\setminus Y)$ is over (see lemma \ref{lemma3.1}). Further, the map 
$f\colon \sph^2\setminus X\to\sph^2\setminus Y$ is a regular covering map.  
Then if 
$M$ is a finite geometric type surface and $F\colon M\to\sph^2$ is a branched 
covering with $Y=\sph^2\setminus F(M)$ and $\sharp Y=3$ we can do a lifting of 
$F$ by $f$ to a new map $\tilde F\colon M\to\sph^2\setminus X$ having 
continuous extension to $\M$ where $X=\sph^2\setminus 
\tilde F(M)$ and $\sharp X=6,$ getting a contradiction with theorem 
\ref{th3.2}.
 It is important to light lemma \ref{lemma3.1} is true only for $Y$ with 
$\sharp Y\geq3.$

 In fact this generalize the little Picard theorem in the case the entire 
function is rational  for a new class of surfaces and for branched covering 
(see \S3 for 
details). We show that any branched covering map $F\colon M\to\sph^2$ with 
continuous extension to $\M$ can not 
miss more then $2$ points unless it is constant, without use of the conformal 
structure ($F$ do not need to be holomorphic nor $M$ to be Riemann surfaces). A 
particular case is that the Gauss map of one finite geometric type surfaces can 
not miss $3$ or more points unless it is constant, generalizing the theorem of 
\cite{kn:J-Mercuri}. We get 
\begin{assertion}[Generalization of little Picard theorem] \label{ftc}
Let $M=\M\setminus E_m$ be a finite geometric type surface and $F\colon 
M\to\sph^2$ be a non constant branched covering with finite fiber that has a 
$C^0$ extension to a branched 
covering $ F\colon\M\to\sph^2.$ Then $\sharp\big(\sph^2\setminus 
F(M)\big)\leq2.$ In particular if $G$ is the Gauss map of $M$ and $M$ not flat  
then $\sharp 
\big(\sph^2\setminus G(M)\big)\leq2.$
\end{assertion}
In \cite{kn:M-S} the authors ask  who are the embedded minimal surfaces of 
finite geometric type whose Gauss map misses two points (question Q1). The same 
question appear in Conjecture 17.0.33, item 3 of \cite{kn:M-P}. In \S 
\ref{sec6} we present examples of Gauss maps that have not continuous 
extensions. Those examples should have at least one end $w\in E$ with $G(w)$ a 
missed point of $G|M$ and order $\beta(w)$ of branch such that $1+\beta(w)$ is 
not divided by $3$ (see \S\ref{sec6}).


\section{Finite Geometric Type.}\label{ftp}

We will recall some facts about the behavior of an immersion of finite 
geometric type near  the ends.  Since the Gauss map is defined at a point 
$w\in E$, we have a
{\em
tangent plane at $w$}, namely $G(w)^{\perp}$. It follows from the above
properties that the image of the immersion of a small punctured neighborhood of
$w$ projects (orthogonally)
onto the complement of a disk in $G(w)^{\perp}$ as a finite covering map of
order $I(w)$.\ The number $I(w)$ is called the {\em geometric index} of $w$,
see \cite{kn:J-M}.

Since the branching points of $G$, i.e. the points of zero Gaussian curvature
and, possibly, the ends,
 are isolated, a punctured neighborhood of such a point $v$ is mapped onto its
image, as a covering map of order $\nu(v)$. The number $\beta(v) = \nu(v) - 1$\
is called the {\em branching number} at $v$. Observe that if $v$ is not a
branching point then $\beta (v)=0$. We have the following topological relations:
\begin{description}
 \item[RH (Riemann-Hurwirtz):] $2\dg(G)=\displaystyle \chi(\M)+
\sum_{w\in \M }\beta(w),$
 \item[TC (Total curvature):] $2\dg(G)=-\chi(\M)+\sharp E+
\displaystyle\sum_{i=1}^{m}I(w_i),$
\end{description}
where $\chi(\M)$\ is the Euler characteristic of $\M$, $I(w_i)$\ is the
geometric index of  $w_i\in E$\ and ${\rm deg}(G)$\ is the degree of the Gauss
map $G$ (the cardinality of $G^{-1}(y)$ for almost all $y$ in the image).

The first relation is a well known fact in covering space theory. The second 
one was obtained, as an inequality, for the case of minimal surfaces, by 
Osserman \cite{kn:O} and in the above form, by Jorge and Meeks \cite{kn:J-M} 
(see also \cite{kn:B-F-M}).\

The two relations {\bf RH}\ and {\bf TC} seems not to be enough to prove the 
sharp bound $\sharp Y \leq 2$, but are enough to assert theorem \ref {ftc} in 
the class of finite geometric type surfaces.  We will use the following 
notations
for one immersion $f$ of finite 
geometric type:
\begin{equation}\label{notations}
\begin{array}{llll}
E^\infty = G^{-1}(Y),&  E_0 = E \setminus E^\infty,&\ell_0=\sharp G(E)\setminus 
Y, \ell = \sharp Y,\\&&&\\
 n = - \deg(G),&\displaystyle\beta(W)=\displaystyle\sum_W\beta(w),& 
I(W)=\displaystyle\sum_WI(w).\
\end{array}
\end{equation}
Since $G\colon\M\to\sph^2$ is a branch covering we have 
\begin{equation}\label{eq3,1}
\ell n\ =\ \sharp E^\infty\ +\ \beta(E^\infty)
\end{equation}
Combining {\bf RH} and {\bf TC} we obtain the following theorem (see also 
\cite{kn:B-F-M})

\begin{theorem}\label{th3.1} If $f: M \lra \R^3$\ is a surface of finite 
geometric type,
then
	\begin{equation}
		0\leq\ell\leq3\label{eq3,2}
	\end{equation}
If $\ell=3$ then $\chi(\M)\leq0.$ In addition if $\ell=3$ and $\chi(\M)=0$ then 
$E=E^\infty,\ \beta(M)=0,\ \beta(E)=2n$, $n=\sharp E=I(E)$ and all ends are 
embedded.   
\end{theorem}
\begin{proof}
	Adding ({\bf RH})\ and ({\bf TC)}\ we get
	
\begin{equation}\label{eq3.2.1}
4n=\sharp E\ +\ \beta(\M)\ +\ I(E)
\end{equation}
Substituting equation (\ref{eq3,1}) we get
\begin{equation}\label{eq3.2.2}	
	4n=\ell n\ +\ \sharp E_0\ +\ \beta(E_0\cup M)\ +\ I(E)
\end{equation}
Hence
\begin{equation} \label{eq3.2.3}
(4-\ell)n = \sharp E_0\ +\ \beta(M\cup E_0)\ +\ I(E)\ \ > 0.
\end{equation}
proving that $0\leq \ell\leq 3.$ Now assume $\ell=3.$  
The ({\bf RH}) and $\beta(E^\infty)+\sharp E^\infty=3n$ gives $\sharp 
E^\infty=\chi(\M)+n+\beta(E_0\cup M).$ Using the ({\bf TC}) we get
\begin{eqnarray*}
2n&=& -\chi(\M)+\sharp E+I(E)\\
&=&n+\sharp E_0+\beta(E_0\cup M)+I(E)\\
&=& 2n+\chi(\M)+2\sharp E_0+2\beta(E_0\cup M)+I_0
\end{eqnarray*}
where $$I_0=\sum_E(I(w)-1)\geq0$$ Then
\begin{equation}
\chi(\M)+2\sharp E_0+2\beta(E_0\cup M)+I_0=0
\end{equation}
giving $\chi(\M)\leq0$ and if $\chi(\M)=0$ then $E_0=\emptyset$ and 
$\beta(M)=0$ proving the result.
\end{proof}
\begin{lemma}\label{lemma3.1}
 Given three points $a_1=\infty,\ a_2=0,\ x_1,\ x_1\notin\{a_1,\ a_2\}$ 
of $\sph^2=\CC\cup\{\infty\},$ the points $x_2=-3x_1,\ y_1=-x_1,\ 
y_2=3x_1,\ w=16x_1^3$ and the rational function 
\begin{equation}\label{eq3.1}
 f(z)=\frac{(z-x_1)^3)(z-x_2)}{z},\quad z\in \sph^2
\end{equation}
then
\begin{equation}\label{eq3.2}
 \deg f=4,\quad \beta_f(a_j)=\beta_f(x_j)=\beta_f(y_j)=\left\{ 
\begin{array}{ll}
                                                               2,& j=1,\\
                                                               0,& j=2
                                                              \end{array}
\right.
\end{equation}
In particular if $X=\{0,\ \infty,\ x_1,\ x_2,\ y_1,\ y_2\}$ then
\begin{equation}\label{eq3.3}
 f\colon\sph^2\setminus X\to\sph^2\setminus\{0,\ w,\ \infty\}
\end{equation}
is a regular covering map of degree $4$ and $\sharp X=6.$ Fixed appropriated 
basics points we have that the map $f_\star$ induced by $f$ between the 
fundamental groups $\pi_1(\sph^2\setminus X)$ and $\pi_1(\sph^2\setminus Y)$ is 
over.

Reciprocally, given a branched covering $h\colon\sph^2\to\sph^2$ and subsets $Y,\ X=h^{-1}(Y)$, such that 
\begin{enumerate}
\item The set $B_h$ of branching points of $h$ is a subset of $X$,
\item $h_\star\big(\pi_1(\sph^2\setminus X)\big)=\pi_1\big(\sph^2\setminus Y\big)$
\end{enumerate} 
 then $\sharp Y=3+m$, $m\geq0$ an integer, $\sharp X=6+4m$, $\beta_h(x)=2$ 
 for all $x\in B_h,\ \deg h=4,$ and $\sharp B_h=3.$ 
 We have splits $X=X_0\cup X_1,\ Y=Y_0\cup Y_1$ with 
 $\sharp X_0=6,\ \sharp Y_0=3,\ \sharp X_1=\sharp Y_1=m,$ and $B_h\subset X_0.$

 Further, if $\psi$ is the M\"obius transform 
such that $\psi(Y_0)=\{0,\ w,\ \infty\},$ and 
$X_f=f^{-1}(\psi(Y)),$ where $f$ is defined by the choose 
$x_1=(w/16)^{1/3},$  then there is a diffeomorphism preserving fiber 
$\phi\colon\sph^2\setminus X_f\to\sph^2\setminus X$ such that   $\psi\circ 
h\circ \phi=f.$  
\end{lemma}
\begin{proof}
 The function $f(z)=(z-x_1)^3(z-x_2)/z,\ z\in\sph^2,$ satisfy the conditions 
for $z\in\{a_j,\ x_j|\ j=1,\ 2\},\ \deg f=4$ and $f(y_1)=f(y_2)=w.$ Since 
\[
 f'(z)=3\frac{(z-x_1)^2(z-y_1)^2}{z^2}, 
\]
it follows that $f$ fulfill all conditions of the lemma. 
Since $\sum_{z\in\sph^2} \beta_f(z)=6$ there is no more branch for $f$ unless 
$\infty,\ x_1$ and $y_1$ proving that $f\colon\sph^2\setminus 
X\to\sph^2\setminus\{0,\ w,\ \infty\}$ is a regular covering. If $\gamma$ is a 
small circle around $x\in X$ then $f(\gamma)$ is a small closed curve around 
$y=f(x)\in \{0,\ w,\ \infty\}$ giving $1+\beta_f(x)$ loops up to orientation. 
Hence the induced map $f_\star$ between the first fundamental groups 
$\pi_1(\sph^2\setminus X$ and $\sph^2\setminus \{0,\ w,\ \infty\}$ satisfies
\[
 f_\star([\gamma_x])=[\delta_y]^{\pm(1+\beta_f(x))},\quad x\in X,\quad 
y\in\{0,\ w,\ \infty\},
\]
where $\gamma_x,\ \delta_y$ are generators. Since each point $y\in\{0,\ w.\ 
\infty\}$ is image of a point $x\in X$ with $\beta_f(x)=0$ we get that all 
generators of $\pi_1(\sph^2\setminus\{0,\ w,\ \infty\})$ belongs to 
$f_\star(\pi_1(\sph^2\setminus X))$, that is,
\[
 \pi_1(\sph^2\setminus\{0, w,\ \infty\})\subset f_\star(\pi_1(\sph^2\setminus 
X)),
\]
completing one way of proof.

Assume now the existence of a branched covering $h\colon\sph^2\to\sph^2$ and sets $Y,\  X=h^{-1}(Y)$ with the branchings of $h$ 
into the set $X$ such that 
$h_\star\big(\pi_1(\sph^2\setminus X)\big)=\pi_1\big(\sph^2\setminus Y\big)$. Set $Y_0=h(B_h),\ X_0=h^{-1}(Y_0),$ and the splits $Y=Y_0\cup Y_1,\ X=X_0\cup X_1.$ 
Hence $h_\star$ send generator into generator implying that $X_0$ has a subset $X_0'$ 
with $\sharp X_0'=\sharp Y_0$ having zero branching points. Then $X_0=X_0'\cup B_h$ and 
$\sharp X_0=\sharp X_0'+ \sharp B_h$.
We have
\[
\sharp Y_0\deg(h)=\sharp X_0+\beta(B_h)=\sharp Y_0+\sharp B_h+\beta(B_h).
\]
Subtracting the $({\bf R-H})$ we get
\begin{equation}\label{eq3.6}
 \Big(\sharp Y_0-2\Big)\Big(\deg(h)-1\Big)=\sharp B_h+\beta(B_h)>0.
\end{equation}
implying that $\sharp Y_0\geq3$ and $\deg (h)>1.$ 
Choose  a M\"obius transform $\psi_0$ such that $\psi_0(Y_0)=\{0,\ w_0,\ \infty \}$ and 
$W=\psi_0(Y)=\{0,\ w_0,\ \infty \}\cup Y'.$ Let $f\colon\sph^2\to\sph^2$ be the 
the map of the first part of the lemma defined by 
$x_1,\ x_1^3=w_0/16.$ If $X_f=f^{-1}(W)$ the map 
$f\colon\sph^2\setminus X_f\to\sph^2\setminus W$ is a 
regular covering and $f_\star$ is over.

Then we have two lifting $\tilde h\colon \sph^2\setminus X\to\sph^\setminus X_f$ of 
$\psi_0\circ h$ by $f$ and $\tilde f\colon\sph^2\setminus X_f\to\sph^2\setminus X$ of $f$ by $\psi_0\circ h.$ Since $h=f\circ\tilde h$ and 
$f=h\circ\tilde f$ we get $\deg\tilde h\cdot \deg\tilde f=1$ implying that both 
are diffeomorphism and $\deg h=\deg f=4,\ \sharp X=\sharp X_f.$ Then 
\[
\sharp f^{-1}(\{0,\ w_0,\ \infty \})=3\deg 
f-\beta_f(f^{-1}(\{0,\ w_0,\ \infty \})))=3\times 4-6=6. 
\]
 and 
\begin{eqnarray}
 \sharp X_f&=&\sharp f^{-1}(\{0,\ w_0,\ \infty \}\cup Y')=f^{-1}(Y')+\sharp 
f^{-1}(\{0,\ w_0,\ \infty \}),\\
 &=&6+\deg f\sharp Y'
\end{eqnarray}
implying that 
\[
 6+4\sharp Y'=\sharp X_f=\sharp X.
\]
\end{proof}
 Let $D_1,\ D_2,$ and $D_3$ be $3$ disks without the center and $f\colon 
\overline{D}_1\to \overline{D}_2$ and $F\colon\overline{D}_3\to\overline{D}_2$ 
branched covering maps with one branching at the center of each disk of order 
$\beta_f$ and $\beta_F.$ We consider $f$ and $F$ of class $C^l,\ l\geq 2,$ 
inside $D_j$ and continuous in $\overline{D}_j.$ The first homotopy group 
$\pi_1(D_j)$ is an infinite  cyclic group with generators $\gamma_j$. The 
existence of a continuous lift $\tilde F\colon D_3\to D_1$ of $F$ by $f$ is 
equivalent to 
\[
 F_\star\big(\pi_1(D_3)\big)\subset f_\star\big(\pi_1(D_1)\big),
\]
where the subindex means the induced group homomorphism between the fundamental 
groups. Since $f_\star[\gamma_1]=[\gamma_2]^{1+\beta_f}$ and 
$F_\star[\gamma_3]=[\gamma_2]^{1+\beta_F}$ the existence of $\tilde F$ is 
equivalent to have 
\[
 f_\star[\gamma_1]^k=F_\star[\gamma_3],\quad k\in \mathbb{Z},\quad k\geq1,
\]
or yet $1+\beta_F=k(1+\beta_f),\ k\geq1.$

\begin{center}
\begin{minipage}{3cm}
\xymatrix{
\tilde{D}_1  \ar[d]_{z^{k}} & D_3 \ar[l]_{\zeta_j} 
\ar[d]^{z^{1+\beta_F}} \ar[dl]_{\tilde F} \\
D_1 \ar[r]^f & D_2 \\
} 
\end{minipage}
  
\end{center}

\begin{lemma}\label{lemma3.2}
 Let $D_1,\ D_2,$ and $D_3$ be $3$ disks without the center and $f\colon 
\overline{D}_1\to \overline{D}_2$ and $F\colon\overline{D}_3\to\overline{D}_2$ 
branched covering maps with one branching at the center of each disk of order 
$\beta_f$ and $\beta_F.$ We consider $f$ and $F$ of class $C^l,\ l\geq 2,$ 
inside 
$D_j$ and continuous in $\overline{D}_j.$ Then there exist a continuous lifting 
$\tilde{F}\colon\overline{D}_3\to\overline{D}_1$ iff  
\begin{equation}\label{eq3.6.1}
 \frac{1+\beta_F}{1+\beta_f}=k=1+\beta_{\tilde{F}},\quad k\in \mathbb{Z},\quad k\geq1.
\end{equation}
In that case there exist a $k$-root $\zeta_j\colon D_3\to {\tilde D}_1$ of 
$\tilde F$ such that $\zeta_j^k=\tilde F$ and $\tilde F$ has the same 
differentiability of $f$ and $F.$ 
\end{lemma}
\begin{proof}
 We can assume that 
 \[
F(z)=z^{1+\beta_F},\   \quad z\in D_3,  
\]
and $f(w)=w^{1+\beta_f}Q(w),$ where $Q(w)\neq0,\ w\in D_1.$ The existence of 
$\tilde F$ is equivalent to $1+\beta_F=k(1+\beta_f)$. Then the degree of 
$\tilde F$ is $k$ implying on the existence of a $k$-root $\zeta_j$. This shows 
that $\tilde F$ is smooth once 
$\zeta_j^{1+\beta_F}=\big(zQ_1(z)\big)^{1+\beta_F}$ and 
$\zeta_j\circ\tau=zQ_1(z)$ where $\tau$ is one deck transform. Then $\tilde F$ 
has the same differentiability of $f.$
\end{proof}

The next result is similar to theorem \ref{th3.1} but is a weak version of the 
little Picard's theorem in the case of rational functions. It means that we can 
change 
the Gauss map with any other branched covering $F\colon M\to\sph^2$ and the 
same conclusion about the number of missing points holds. It is consequence 
that {\bf (RH)} and {\bf (TC)} formulas are true for $F.$
\begin{definition}
Let $F\colon M\to\sph^2$ be a $C^2$ branching covering that has $C^0$ extension 
to $F\colon\M\to\sph^2$. Let $Z\subset \M$ be the subset where this extension 
is made. If $Z\neq\emptyset$ we define the order of the branching 
$\beta_F(z),\ z\in Z,$ by
\[
 \beta_F(z)=\lim_{\epsilon\to0}\frac{1}{2\pi}\int_{\partial D_\epsilon}k_g-1
\]
where $D_\epsilon\subset\M$ are small neighborhoods of the end $z\in Z\subset 
E_m$ endowed with the metric of $\sph^2$ by $F$ and the geodesic curvature with 
respect to this metric. 
\end{definition}

\begin{theorem}\label{th3.2}
 Let $M=\M\setminus E_m$ be a finite geometric type surface. 
Let $F\colon M\to\sph^2$ be an at least $C^2$ branched covering map having 
$C^0$ extension 
to a branched covering map denoted by $F\colon\M\to\sph^2.$ 
Then the ({\bf T-C}) and ({\bf R-H}) formulas holds
\begin{eqnarray}
 2\deg F&=&-\chi(\M)+\sharp E_m+I(E_m),\label{eq3.13}\\
 2\deg F&=&\chi(\M)+\beta_F(\M). \label{eq3.13.0}
\end{eqnarray}
In particular  $$\sharp Y\leq 3$$ for  
$Y=\sph^2\setminus F(M)$. 
\end{theorem}
\begin{proof}
 Set $Z=F( E_m),$ where $E_m$ is the set of ends of $M.$ Consider  
$W=F^{-1}(Z).$ Over $\M\setminus W$ take one appropriated $\epsilon>0$ such 
that:
\begin{enumerate}
 \item All the balls $B_\epsilon^{\sph^2}(z),\ z\in Z,$ are 
disjoints
\item The collection $U_\epsilon=F^{-1}(B_\epsilon^{\sph^2}(z)\setminus\{z\}),\ 
F(w)=z\in Z,\ w\in W,$ are disjoints topological annulus isolating the points 
of $W$ and $\text{diam}(U_w)\mapsto0$ if $\epsilon\mapsto0$ where $U_w$ are the 
connected component of $U_\epsilon.$
\end{enumerate}
Set $V_\epsilon=\{B_\epsilon^{\sph^2}(z)\ |\ z\in Z\}.$ Observe 
that $\M\setminus W\subset M$ and $F|(M\setminus W)$ is a regular covering map. 
If we endowed $\sph^2\setminus Z$ with the metric of $M$ we get
\begin{eqnarray}
 2\deg(F)-\text{area}(V_\epsilon)&=&-\frac{1}{2\pi}\int_{M\setminus 
U_\epsilon}KdM\\
 &=&-\big(\chi(\M)-\sharp W\Big)+\frac{1}{2\pi}\int_{\partial(M\setminus 
U_\epsilon)}k_g,\\
&=&-\chi(\M)+\sharp W-\frac{1}{2\pi}\int_{\partial U_\epsilon}k_g,
\end{eqnarray}
where $K$ is the Gaussian curvature of $M$ and $k_g$ is the geodesic curvature 
of $\partial U_\epsilon.$ Hence 
\begin{equation}\label{eq3.7}
 \frac{1}{2\pi}\int_{\partial U_\epsilon}k_g=\sum_{w\in 
W}\frac{1}{2\pi}\int_{\partial U_w}k_g.
\end{equation}
If in some disk $D_\epsilon=B_\epsilon^{\R^2}(0)$ we have a metric 
$\lambda|dz|=|z|^s\nu(z)|dz|,$ where $\nu(z)$ has no zeros, a straightforward 
calculation gives 
\begin{equation}\label{eq3.8}
 \lim_{\epsilon\to0}\int_{\partial D_\epsilon}k_g=2\pi(1+s). 
\end{equation}
For isothermal parameter $\psi\colon D_\epsilon\to M\subset\M$ with 
$\psi(0)=w\in W$ and $w\in E_m$ we get $s=-1-I(w)$ where $I(w)$ is the 
geometric index of the end. Hence 
\begin{equation}\label{eq3.9}
 \lim_\epsilon\int_{\gamma_\epsilon}k_g=-2\pi I(w),
\end{equation}
If $w\in W\cap M$ then the 
metric has $s=0$ and   
\begin{equation}\label{eq3.10}
 \lim_{\epsilon\to0}\frac{1}{2\pi}\int_{\partial 
U_w}kg=1.
\end{equation}
Making $\epsilon\mapsto0$ we see that the {\bf TC} formula holds for $F$, or
\begin{eqnarray}
 2\deg F&=&-\chi(\M)+\sharp W-\sharp(M\cap W)+I(E_m),\\
&=&-\chi(\M)+\sharp E_m+ I(E_m). 
\end{eqnarray}
As we did before, if $Y=\sph^2\setminus F(M)$ then $F^{-1}(Y)\subset E_{m}$ 
once that $F$ extends to the branching  and ends. Then 
\begin{equation}\label{eq3.11}
 \sharp Y\deg(F)=\sharp E_m^\infty+\beta_F(E_m^\infty),\quad E_m^\infty=E_m\cap 
F^{-1}(Y).
\end{equation}
In similar way the map $F\colon M\setminus W'\to\sph^2\setminus F(E_m\cup B),\ 
B$ the set of branching, is an immersion and we can do the pull back of the 
metric of $\sph^2$ by $F.$
Following the proof of (\ref{eq3.13}) we get 
\begin{eqnarray}
  2\deg( F)-\text{area}(V_\epsilon)&=&\frac{1}{2\pi}\int_{M\setminus 
U_\epsilon}KdM\\
 &=&\chi(\M)-\sharp W-\frac{1}{2\pi}\int_{\partial(M\setminus 
U_\epsilon)}k_g,\\
&=&\chi(\M)-\sharp W+\frac{1}{2\pi}\int_{\partial U_\epsilon}k_g,
\end{eqnarray}
where $K\equiv1$ is the Gaussian curvature and $k_g$ is the geodesic curvature 
in the metric $\lambda(F)|dF||dz|$ and $\lambda(w)|dw|$ is the metric 
of the sphere. The equations (\ref{eq3.7}) and (\ref{eq3.8}) holds and 
\begin{equation}
 s=
           \beta_F(x),\qquad x\in B.
\end{equation}
Hence
\begin{equation}\label{eq3.16.0}
 2\deg F=\chi(\M)+\beta_F(\M). 
\end{equation}
The equation $\sharp W\deg F=\sharp F^{-1}(W)+\beta_F(V)$ is true if $W\cap F(E_m)=\emptyset.$ Since this 
number is one integer and constant over $\sph^2\setminus F(Z)$ it is the same 
over all $W.$
Adding the {\bf TC} and {\bf RH} formulas for $F$ we get
\begin{equation}\label{eq3.12}
 (4-\sharp Y)\deg(F)=\sharp(E_{m}\setminus E_m^\infty)+\beta_F(\M\setminus 
E_m^\infty)+I(E_m)>0,
\end{equation}
implying that $ \sharp Y\leq3$ as we wish.
\end{proof}

\begin{remark}\label{remk2}
Let $M=\M\setminus E_m$ a finite geometric type surface and $G\colon 
M\to\sph^2$ a branched covering map with $C^0$ extension to 
$G\colon\M\to\sph^2$ and not empty set of missing points $\sph^2\setminus 
G(M).$
Assume there are a regular covering map $h\colon\sph^2\setminus 
X\to\sph^2\setminus Y$ such that 
\begin{equation}\label{eq3.13.1}
 h_\star(\pi_1(\sph^2\setminus X))=\pi_1(\sph^2\setminus Y),\quad \sph^2\setminus G(M)\subset Y.
\end{equation}
Let $F\colon M\to\sph^2\setminus X$ the lifting of $G$ by $h.$ 
\begin{question}
When the map $F$ can be extended continuously to $\M$?
\end{question}
For all point $x\in h^{-1}(y)$ where $\beta_h(x)=0$ the map $F$ can be extended to one point in the fiber $G^{-1}(y)$ smoothly. For $y\in Y$ and $x\in h^{-1}(y)$ with $\beta_h(x)>0$ we need the condition given  by lemma (\ref{lemma3.2}).
If it do not happen there is no $C^0$ extension. 
\end{remark}

\subsection{Proof of theorem 3.}
Let $M=\M\setminus E_m$ be a finite geometric type surface and $G\colon 
M\to\sph^2$ be a non constant branched covering with finite fiber that has a 
$C^0$ extension to a branched 
covering $ G\colon\M\to\sph^2.$ Set $Y=\big(\sph^2\setminus 
G(M)\big)$ and suppose that $\sharp Y=3.$ Hence it is possible to find 
$h\colon \sph^2\setminus X\to\sph^\setminus Y$ satisfying (\ref{eq3.13.1}) and 
proving the existence of the lifting $F\colon M\to\sph^2\setminus X$ of $G$ by 
$h$ and $F$ has continuous extension to $\M$. Since $X=\sph^2\setminus F(M)$ 
and $\sharp X=6$ we get a contradiction with theorem (\ref{th3.2}) 
proving the theorem 3.


\section{The proof of theorem 1.}
The space finite geometric type surfaces is closed for an operation of bending 
the ends.
Take $f\colon M\to\R^3$ in FGT class with Gauss map $G$ and set of ends 
$E_k=\{w_1,\cdots,w_k\}.$ By \cite{kn:J-M} 
we can choose $R_0>0$ such that $f$ satisfies 
\begin{enumerate}
\item[(i)] $f^{-1}(\R^3\setminus( B_{R}^{\R^2}(0)\times\R))=\cup_{j} V_{jiR},$ 
where
 $V_{jiR}$ is a neighborhood of each end $w_j\in E,\ G(w_j)=y_i\in Y,$
in
$\M,$ for all $R\geq R_0$ and all $i,$
\item[(ii)] $F_i=P\circ f_i\colon (V_{jiR}\setminus\{w_j\})\to \R^2\setminus
B_{R}^{\R^3}(0)$ is a covering with fibre's cardinality the geometric index 
$I(w_j),$ for all
$R\geq R_0$ and for all $w_j\in E$ and all $i.$
\end{enumerate}

Choose $\delta_0>0$ such that \[B_{\delta_0}^{\sph^2}(y_j)\subset
G(V_{jR_0}),\qquad y_j=G(w_j),\quad  w_{j} \in E,\]
and all balls disjoints.
There is $R>R_0$ such that
\begin{equation}
G(V_{jR})\subset B_{\delta_0}^{\sph^2}(y_j),\qquad w_j \in E.
\end{equation}
Fix some $y_j\in G(E)$ and take $Z=\cup V_{jR}'$ where $\{w_j'\}= V_{jR}'\cap
E$ satisfies $G(w_j')=y_j.$  Now we are in position to {\em bend} the set $Z$
to get a new $C^3$ immersion $\tilde{f}$ that just move $y_j$ to a point $y$
close to $y_j.$
Choose some $\delta_1>0$ such that
\[ \sin(\delta_1)<\text{dist}_{\sph^2}\big(y_j,f(\partial Z)\big)/4\]
and $y\in B_{\delta_1}^{\sph^2}(y_j).$
Take a $C^\infty$ function $\psi\colon\R^3\to\R,\ 0\leq\psi\leq 1,$ with
$\psi\equiv 0$ for $|x|\leq 2R$ and $\psi\equiv 1$ for $|x|\geq 3R.$  Define
$h\colon\R^3\to\R^3$ by
$h(x)=(1-\psi(x))x+\psi(x)A_\delta x$ where $\delta$ is the angle between $y_j$
and $y$ and $A_\delta$ is the rotation of angle $\delta>0$ moving $y_j$ to $y.$
Let $\tilde{f}\colon M\to\R^3$ be defined by $\tilde{f}(x)=f(x)$ if $x\in
M\setminus Z$ and $\tilde{f}(x)=h(f(x))$ for $x\in Z.$
Since $f$ is proper the set $K=f^{-1}\big(B_R^{\R^2}(0)\big)$ is compact in
$M.$ If we take $\delta_1$ small enough with no points with zero Gaussian 
curvature in the gluing area. Hence we have the following result.
\begin{proposition}\label{prop-spacefuntion} Let $f$ a finite geometric type 
surface with projection $F=P\circ f$ and let
$\tilde{f}$ be the bending constructed above. Let $G$ and $\tilde{G}$ be the
respective Gauss maps. Then
\begin{enumerate}
\item $G(x)=\tilde{G}(x)$ for all $x\notin Z$ and $\tilde{G}(Z)\subset
B_{\delta_1}^{\sph^2}(y_j),$
\item $\tilde{G}(w_j')=y$ for all $w_j'\in G^{-1}(y_j)$ for any chose of $y\in
B_{\delta_1}^{\sph^2}(y_j),$
\item if $G$ misses $y_j$ then $\tilde{G}$ misses $y,$
\item $G(w)=\tilde{G}(w)$ for all $w\in E\setminus G^{-1}(y_j),$
\item the number of missing points of $G$ and $\tilde{G}$ are the same,
\item the geometric index of $f$ and $\tilde{f}$ are the same at each end,
\item the order of branching points of $G$ and $\tilde G$ are the same.
\end{enumerate}
\end{proposition}
The above results proves Theorem 1.

\section{Gauss map not branched: proof of theorem 2, items 2 and 
3.}\label{cov-map}
Let $M=\M\setminus E_m$ a finite geometric type surface with $\M$ of genus 
$\mu.$ We can represent $\M$ as one regular rectangle of $\mathbf{R}_\mu$ of 
$4\mu$ sides  with boundary ordered by the relation $a_ib_ia_i^{-1}b_i^{-1},\ 
1\leq i\leq\mu.$ The vertexes are identified and the sides $a_i$ and $b_i$ 
defining  curves $\alpha_i$ and $\gamma_i$, witch are generators of the 
fundamental 
group $\pi_1(\M).$ When we take $M=\M\setminus E_m$ we add more $m-1$ 
generators 
to the first fundamental group. From now on we assume that the Gauss map 
\[G\colon M\to\sph^2\setminus Y,\qquad Y=\sph^2\setminus G(M)\]
 is a regular covering map, that is, there are no branching. In particular we 
have $E_m^\infty=E_m,\ E_m^0=B_0=\emptyset,$ and $B^\infty\subset E_m.$ The 
{\bf R-H} formula gives
\begin{equation}\label{eq2.1}
n(2-\sharp Y)=\chi(M)=\chi(\M)-\sharp E_m,
\end{equation}
or
\begin{equation}\label{eq2.2}
\chi(\M)=\sharp E_m+n(2-\sharp Y).
\end{equation}
If $\sharp Y=1$ then $\sharp E_m+n=2$ and $M=\sph^2\setminus\{\infty\}$ 
proving item 3 of theorem 2. If $\sharp Y=2$ we get $\sharp E_m=2$ proving item 
2 of theorem 2.

\section{New proof of a Lopes-Ros theorem.}
In \cite{kn:L-R} the authors proved the following result.
\begin{theorem}[Lopes-Ros, \cite{kn:L-R}]
 Let $M$ be a non flat complete minimal surface with finite total curvature and 
embedded. If $M=\sph^2\setminus E_m$  then $M$ is a catenoid.
\end{theorem}
The original proof of this theorem use deformation of small pieces of $M$ under 
some special flow to get the conclusion. Theorem 3 implies this result easily. 
In fact, consider $M$ in some position that $\infty$ is the image of some end. 
Then we can consider $M=\CC\setminus E$ and $\sharp E<\infty.$ If $x\colon 
M\to\R^3$ is such that minimal embedded the map then $f=x^{-1}\colon M\to\CC$ 
is well defined and is a branched covering. By theorem 3 we get 
$\sharp\CC\setminus f(M)\leq 1,$ or $f$ is constant. Hence $\sharp E_m=2$ and 
classical results grantee that $M$ is  a catenoid.

\section{Remarks on Gauss map missing two points.}\label{sec6}
Let $M=\M\setminus E_m$ be a finite geometric type minimal surface whose Gauss 
map $G$ misses two points $\sph^2\setminus G(M)=\{a_1,\ a_2\}.$ 
The existence of those examples are guarantee 
by \cite{kn:M-S} for genus $1$ and \cite{kn:G-J} for arbitrary genus but all 
are not embedded. Since $2\deg(G)=\sharp E_m^\infty+\beta(E_m^\infty)$ the {\bf 
R-H} formula gives 
\begin{equation}
 \chi(\M)=\sharp E_m^\infty-\beta(M\cup E_m^0),
\end{equation}
where $E_m^\infty=G^{-1}(\{a_1,a_2\})$ and $E_m^0=E_m\setminus E_m^\infty.$
Take $w\in\sph^2\setminus \{a_1,\ a_2\}$, $Y=\{a_1,\ a_2,\ w,\}$ and 
$f\colon\sph^2\setminus X\to\sph^2\setminus Y$ the map of lemma 
(\ref{lemma3.1}) with $X=f^{-1}(Y).$ Then $f|\sph^2\setminus X$ is a regular 
covering and there is the lifting $F\colon M\setminus 
G^{-1}(Y)\to\sph^2\setminus X$. By Remark \ref{remk2} we have $C^0$ extension of
$F$ to $\M$ if and only if $3$ divides all 
$1+\beta_G(z)\geq2,\ z\in E_m^\infty\cup G^{-1}(w).$
 But $\sharp \sph^2\setminus F(M)=4$ what is impossible by theorem \ref{th3.2}.  
Then for those immersions we can not have 
continuous extensions $ F\colon\M\to\sph^2$ for any $Y\supset \{a_1,\ 
a_2\},\ \sharp Y\geq3,$ no matter $G$ has.


\end{document}